\documentclass{article}

\usepackage{amsmath}
\usepackage{amssymb}
\usepackage{amsthm}

\usepackage{pgf}
\usepackage{tikz}

\usepackage{url}

\theoremstyle{plain}
\newtheorem{Thm}{Theorem}
\newtheorem{Prop}[Thm]{Proposition}

\newtheorem{Lemma}[Thm]{Lemma}

\theoremstyle{definition}

\newtheorem*{Remark*}{Remark}

\newtheorem{Def*}{Definition}

\definecolor{checkersColor}{RGB}{255,60,60}
\definecolor{pathColor}{RGB}{0,0,240}
\definecolor{firstColor}{RGB}{0,0,240}
\definecolor{secondColor}{RGB}{0,180,0}

\title{Billiards, Checkers, and Quadratic Reciprocity}

\author{Johan Wästlund}
\date{} 

\begin{document}

\maketitle

\begin{abstract}
We indulge in what mathematicians call frivolous activities. In Arithmetic Billiards, a ball is bouncing around in a rectangle. In Parity Checkers we place checkers on a checkerboard under certain parity constraints. Both activities turn out to capture the division of congruence classes modulo a prime into squares and non-squares, allowing fairly simple proofs of the celebrated Law of Quadratic Reciprocity. 

Since the activities are analyzed somewhat in parallel we don't obtain two independent proofs. But Franz Lemmermeyer's online list of reciprocity proofs already contains well over three hundred items, which seems enough anyway.
\end{abstract}

\section{Two activities} \label{S:intro}
In Arithmetic Billiards \cite{pool, P}, a ball is moving in a rectangle of integer sides $m$ and $n$. It starts in the lower left corner and moves at 45 degree angles to the sides of the rectangle, bouncing in the obvious way when it hits the boundary. Eventually the ball reaches another corner and that completes the path.

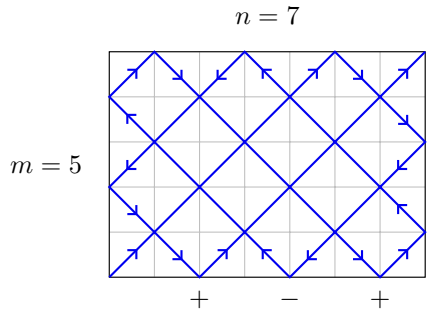
\begin{center}
\begin{tikzpicture} [scale=0.6]

\draw [lightgray] (0,0) grid (7,5);
\draw [black] (0,0) rectangle (7,5);

 \draw[pathColor, thick] (0,0)--(5,5)--(7,3)--(4,0)--(0,4)--(1,5)--(6,0)--(7,1)--(3,5)--(0,2)--(2,0)--(7,5);    
 
 \foreach \x/\y in {0/0, 4/4, 0/4, 6/0, 2/0, 6/4}
  \draw[pathColor, thick] (\x +0.6,\y+0.4) -- (\x +0.6,\y+0.6) -- (\x +0.4,\y+0.6);  
  
  \foreach \x/\y in {6/2, 4/0, 2/4, 0/2}
  \draw[pathColor, thick] (\x +0.4,\y+0.6) -- (\x +0.4,\y+0.4) -- (\x +0.6,\y+0.4);
  
  \foreach \x/\y in {5/4, 6/3, 1/4, 5/0, 0/1, 1/0}
  \draw[pathColor, thick] (\x +0.4,\y+0.4) -- (\x +0.6,\y+0.4) -- (\x +0.6,\y+0.6);
  
  \foreach \x/\y in {3/0, 0/3, 6/1, 3/4}
  \draw[pathColor, thick] (\x +0.4,\y+0.4) -- (\x +0.4,\y+0.6) -- (\x +0.6,\y+0.6);
 
 \node at (-1.4, 2.5) {$m=5$};
 \node at (3.5,5.8) {$n=7$};
 \node at (2,-0.5) { $+$};
 \node at (4,-0.5) { $-$};
 \node at (6,-0.5) { $+$};

\end{tikzpicture}
\end{center}

We say that a bounce along the base of the rectangle is positive if the ball is moving from left to right as it bounces, and negative if it moves from right to left. 
In Section~\ref{S:billiards} (Theorem~\ref{T:legendre}) we show that if the width $n$ is a prime number and the height $m$ is not divisible by $n$, then the number of negative bounces is even if $m$ is congruent to a square modulo $n$, and odd otherwise.

\medskip

Parity Checkers is played on a rectangular checkerboard where the lower left corner is a dark square. For reasons that will become clear we want the checkerboard to have $m-1$ rows and $n-1$ columns. We put black pebbles on some designated set of light squares. An example of particular interest is the set of light squares in the bottom row. 
The challenge is then to put checkers, say red ones, on a set of dark squares in such a way that the light squares that have an odd number of checkers in their immediate neighborhood are precisely those that have pebbles on them. In the 4 by 6 case corresponding to $m=5$ and $n=7$, there is a unique solution:

\medskip

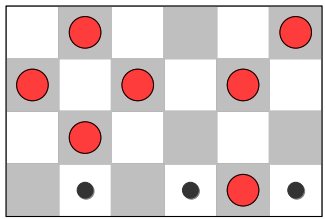
\begin{center}
\begin{tikzpicture} [scale=0.7]

\draw [gray] (0,0) grid (6,4);

\foreach \x/\y in {0/0, 1/1, 2/0, 4/0, 3/1, 5/1, 0/2, 2/2, 4/2, 1/3, 3/3, 5/3}
  \filldraw[lightgray] (\x,\y) rectangle (\x+1,\y+1);
  
\draw (0,0) rectangle (6,4);  

\foreach \x/\y in {0/2, 1/1, 1/3, 2/2, 4/0, 4/2, 5/3}
{
\filldraw [checkersColor] (\x+0.5, \y+0.5) circle (0.3cm);
\draw [black] (\x+0.5, \y+0.5) circle (0.3cm);
};

\foreach \x/\y in {1/0, 3/0, 5/0}
{
\filldraw [gray] (\x+0.52, \y+0.48) circle (0.15cm);
\filldraw [black!80!white] (\x+0.5, \y+0.5) circle (0.15cm);
}
          
\end{tikzpicture}
\end{center}

Counting modulo 2, the configuration of pebbles can be regarded as a linear function of the configuration of checkers. In Section~\ref{S:checkers} we show that this function is invertible on an $(m-1)$ by $(n-1)$ checkerboard precisely when $m$ and $n$ are relatively prime (Propositions~\ref{P:checkersInverse1} and \ref{P:checkersInverse2}). 

Moreover, if $n$ is a prime number that doesn't divide $m$, and we put pebbles on the light squares of the bottom row only, then Theorem~\ref{T:checkersLegendre} shows that the solution to the puzzle will have an even or odd number of checkers according as $m$ is a square or not modulo $n$.

\medskip

{\bf Acknowledgment.} In a seminar talk in Gothenburg in October 2023, Frederik Ravn Klausen discussed edge percolation on $\mathbb{Z}^2$ under parity constraints, for instance requiring every vertex to have even degree. It made me wonder when an edge configuration on a finite grid is determined by local parity constraints, which eventually led to the results presented here.

\section{Number theoretical prerequisites} \label{S:numtheory}

In this section we review some number theory before returning to our games and pastimes. 

If $p$ is an odd prime, the nonzero congruence classes modulo $p$ are divided into equally many quadratic residues and non-residues. The Legendre symbol $\left(a|p\right)$ is defined to be 1 if $a$ is not divisible by $p$ but congruent modulo $p$ to a square, $-1$ if $a$ is not congruent to a square, and 0 if $a$ is divisible by $p$. We can distinguish these cases by a criterion of Leonhard Euler:

\begin{Lemma} [Euler's criterion] \label{L:euler}
\[ a^{(p-1)/2} \equiv \left(\frac{a}{p}\right) \quad \text{\rm (mod $p$)}.\]
\end{Lemma}

\begin{proof}
When $a\equiv b^2$ is a quadratic residue, Fermat's theorem shows that 
\[ a^{(p-1)/2} \equiv b^{p-1}\equiv 1.\]
On the other hand the polynomial equation $x^{(p-1)/2} \equiv 1$ cannot have more than $(p-1)/2$ roots modulo $p$, and therefore if $a$ is not a quadratic residue, $a^{(p-1)/2}\not\equiv 1$. Since nevertheless $a^{p-1}\equiv 1$, the only remaining possibility is that $a^{(p-1)/2}\equiv -1$.
\end{proof}

An argument avoiding Lagrange's theorem on roots of polynomials is given in the classic textbook of Hardy and Wright \cite{HW}: Consider the congruence $xy\equiv a$ (mod $p$). If $a$ is not a square, the set of solutions constitute a grouping of the nonzero residues modulo $p$ into $(p-1)/2$ pairs, each having the product $a$. Multiplying all the pairs we get
\[ (p-1)! \equiv a^{(p-1)/2} \quad \text{(mod $p$)}.\]
 If on the other hand $a$ is a nonzero square, the pairing will leave out two congruence classes $x$ and $-x$ such that $x^2=(-x)^2\equiv a$. In that case since $x \cdot (-x) \equiv -a$, we get
 \[ (p-1)! \equiv -a^{(p-1)/2} \quad \text{(mod $p$)}.\]
Taking $a=1$, which is obviously a square, we see that $(p-1)!\equiv -1$ (mod $p$), which concludes the proof. 
 
\medskip

The Law of Quadratic Reciprocity has received considerable attention since the first proof by Carl Friedrich Gauss, published in 1801 but dated 1796 in his diary. Franz Lemmermeyer lists more than 300 published proofs on his website \cite{L1}, see also his book \cite{L2}. 
\begin{Thm} [Law of Quadratic Reciprocity] \label{T:qr}
If $p$ and $q$ are distinct odd primes,
\[ \left(\frac{p}{q}\right)\left(\frac{q}{p}\right) = (-1)^{(p-1)(q-1)/4}\]
\end{Thm}
In other words, $\left(p|q\right) = \left(q|p\right)$ if at least one of $p$ and $q$ is congruent to 1 modulo 4, and $\left(p|q\right) = -\left(q|p\right)$ if both are congruent to 3 modulo 4.

In Sections~\ref{S:reciprocity} and \ref{S:checkers} we give proofs of Theorem~\ref{T:qr} based on arithmetic billiards and parity checkers respectively.

\section{Arithmetic billiards and the Legendre symbol} \label{S:billiards}

Returning to arithmetic billiards, let's put the $m$ by $n$ rectangle in a coordinate system with the lower left corner at the origin as in Figure~\ref{F:billiards1}, and suppose the ball is moving at a speed of one diagonal of a unit square per unit of time.

\begin{figure}[ht]
\begin{center}
\begin{tikzpicture} [scale=0.6]

\draw [lightgray] (0,0) grid (7,5);
\draw [black] (0,0) rectangle (7,5);

\draw[black, ->=stealth] (0,0)--(8,0);
\node at (8.5,0) {$x$};

\draw[black, ->=stealth] (0,0)--(0,5.8);
\node at (0,6.3) {$y$};

\node at (-1.4, 2.5) {$m=5$};
 \node at (3.5,5.8) {$n=7$};

 \draw[pathColor, thick] (0,0)--(5,5)--(7,3)--(4,0)--(0,4)--(1,5)--(6,0)--(7,1)--(3,5)--(0,2)--(2,0)--(7,5);    
 
 \foreach \x/\y in {0/0, 4/4, 0/4, 6/0, 2/0, 6/4}
  \draw[pathColor, thick] (\x +0.6,\y+0.4) -- (\x +0.6,\y+0.6) -- (\x +0.4,\y+0.6);  
  
  \foreach \x/\y in {6/2, 4/0, 2/4, 0/2}
  \draw[pathColor, thick] (\x +0.4,\y+0.6) -- (\x +0.4,\y+0.4) -- (\x +0.6,\y+0.4);
  
  \foreach \x/\y in {5/4, 6/3, 1/4, 5/0, 0/1, 1/0}
  \draw[pathColor, thick] (\x +0.4,\y+0.4) -- (\x +0.6,\y+0.4) -- (\x +0.6,\y+0.6);
  
  \foreach \x/\y in {3/0, 0/3, 6/1, 3/4}
  \draw[pathColor, thick] (\x +0.4,\y+0.4) -- (\x +0.4,\y+0.6) -- (\x +0.6,\y+0.6);
 
 \node at (2,-0.5) {$+2$};
 \node at (4,-0.5) {$-4$};
 \node at (6,-0.5) {$+6$};
 
\end{tikzpicture}
\end{center}
\caption{Arithmetical billiards on a 5 by 7 rectangle. The bounces along the base occur at times 10, 20, and 30, which are congruent modulo 7 to $-4$, $+6$, and $+2$ respectively.}
\label{F:billiards1}
\end{figure}
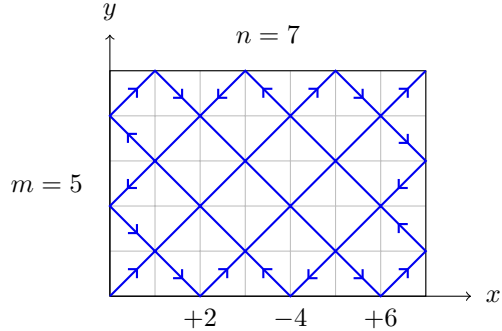

Looking at the $x$-and $y$-coordinates one at a time, the ball is just moving back and forth. This has the following consequences:

\begin{itemize}

\item The top-bottom bounces occur at times divisible by $m$, and the left-right bounces occur at times divisible by $n$. The ball reaches a corner at time equal to the least common multiple of $m$ and $n$.

\item The bounces along the $x$-axis occur at times $2m, 4m, 6m$ and so on. If $\gcd(m,n)=1$ so that $\operatorname{lcm}(m,n) = mn$, the ball will reach all points along the base that are at even distance from the origin.

\item 
Letting $x(t)$ be the $x$-coordinate of the ball at time $t$, we have
\[
\begin{cases}
\text{$x(t)\equiv t$ (mod $n$) when the ball is moving from left to right,}\\
\text{$x(t)\equiv -t$ (mod $n$) when the ball is moving from right to left.}
\end{cases}
\]
\end{itemize}
This last observation lets us conclude the following (again see Figure~\ref{F:billiards1}):

\begin{Lemma} \label{L:timeAndSpace}
If we label each bounce along the $x$-axis with its $x$-coordinate and its sign, then the label of a bounce will be congruent modulo $n$ to the time at which it occurs.
\end{Lemma}

As promised in Section~\ref{S:intro}, we can now relate the signs of the bounces to the Legendre symbol. The following theorem can be thought of as a geometric version of the so-called Gauss's lemma. 

\begin{Thm} \label{T:legendre}
Suppose we play arithmetic billiards on an $m$ by $n$ rectangle where $n$ is a prime and $m$ is not divisible by $n$. Then the product of the signs of the bounces along the base will be equal to the Legendre symbol $(m|n)$. 
\end{Thm}

In other words the number of negative bounces will be even if $m$ is a quadratic residue modulo $n$, and odd if $m$ is a non-residue.
 
\begin{proof}
By Lemma~\ref{L:timeAndSpace}, for each bounce along the $x$-axis, the time at which it occurs is congruent modulo $n$ to the $x$-coordinate times the sign. We multiply these congruences for all bounces along the $x$-axis, and get
\[ (2m)(4m)(6m)\cdots ((n-1)m) \equiv (-1)^s \cdot 2\cdot 4\cdot 6\cdots (n-1)\quad \text{(mod $n$)},\]
where $s$ is the number of negative bounces. 
When $n$ is a prime number we can cancel the factors $2\cdot 4 \cdot 6\cdots (n-1)$ from both sides and get 
\[m^{(n-1)/2} \equiv (-1)^s \quad \text{(mod $n$)}.\] In view of Euler's criterion (Lemma~\ref{L:euler}) this completes the proof.
\end{proof}

At this point we can establish the so-called supplements to the law of reciprocity: 
\begin{Prop} \label{P:supplement} 
When $n$ is an odd prime,
\begin{equation} \notag
\left(\frac{-1}{n}\right) = 
\begin{cases}
1, & \text{\rm if $n\equiv 1$ (mod $4$),}\\
-1, & \text {\rm if $n\equiv 3$ (mod $4$),} 
\end{cases}
\end{equation}
and
\begin{equation} \notag
\left(\frac{2}{n}\right) = 
\begin{cases}
1, & \text{\rm if $n\equiv \pm 1$ (mod $8$),}\\
-1, & \text {\rm if $n\equiv \pm 3$ (mod $8$).} 
\end{cases}
\end{equation}
\end{Prop}

\begin{proof}
We have defined arithmetical billiards only for positive height and width of the rectangle, but $(-1|n)=(n-1\mathbin{|}n)$. If the height is $n-1$ and the width $n$, all bounces along the base will be negative, and the number of such bounces will be even when $n\equiv 1$ (mod 4) and odd when $n\equiv 3$ (mod 4).
 
If on the other hand the height is 2, the bounces taken from left to right will alternate between positive and negative, with negative bounces at $(2k,0)$ whenever $k$ is odd. This means that there will be an even number of negative bounces precisely when $n$ is congruent to $1$ or $-1$ modulo 8.
\end{proof}

\section{Reciprocity} \label{S:reciprocity}
The Law of Quadratic Reciprocity (Theorem~\ref{T:qr}) relates $(m|n)$ and $(n|m)$ when $m$ and $n$ are distinct odd primes. But since we now have an interpretation of the Legendre symbol in terms of signs of bounces, we can prove reciprocity without assuming that $m$ and $n$ are primes.

We extend the definition of the Legendre symbol to arbitrary positive $m$ and $n$ as follows: If $m$ and $n$ are relatively prime, we let $(m|n)$ be the product of the signs of the bounces along the $x$-axis in $m$ by $n$ arithmetic billiards. If $m$ and $n$ have a common factor, we set $(m|n)=0$. We have four remarks on this definition: 

\begin{enumerate}
\item 
If $m$ and $n$ have a common factor, there will be points along the $x$-axis at even distance from the origin that are never reached. Naturally these non-bounces can be interpreted as bounces of sign zero.

\item Strictly speaking we should have denoted our symbol by something like $(m|n)_B$ with an index $B$ for bounces or billiards, since the so-called Jacobi symbol is already an established generalization of the Legendre symbol. But we will keep our notation simple. Our definition actually agrees with the Jacobi symbol for odd ``denominators'', but we don't use this fact.

\item We allow $m$ and $n$ to be even as well as odd. When both are odd, the path will end at the point $(m,n)$, but otherwise it may end either at $(n,0)$ or at $(0,m)$. The endpoint of the path doesn't count as a bounce, but if it did it would be positive, so it wouldn't matter for our definition of $(m|n)$.

\item When the ``denominator'' $n$ is even, there is already a clash between definitions of $(m|n)$ in the literature. Our symbol $(m|n)$ agrees with the Zolotarev definition based on the sign of the permutation $x\mapsto mx$ of the congruence classes modulo $n$ (see \cite{BC}, the postscript of \cite{C}, and the supplement \cite{Montgomery} to \cite{nzm})
But this is not the same as the so-called Kronecker symbol. For example we insist that $(5|8)=1$ even though according to the Kronecker definition, $(5|8)$ would be $-1$. This is crucial in Proposition~\ref{P:almostReciprocity} and Lemma~\ref{L:mod4}, which do not hold for the Kronecker symbol.
\end{enumerate}

Since we are going to compare the two symbols $(m|n)$ and $(n|m)$, naturally we want to consider arithmetic billiards on an $m$ by $n$ rectangle and relate the bounces along the $y$-axis to those along the $x$-axis. Therefore we extend our definition of the sign of a bounce and say that a bounce along any of the two vertical sides of the rectangle is positive if the ball is moving upwards as it bounces, and negative if the ball is moving downwards. Similarly we say that a bounce along the top of the rectangle is positive if the ball is moving to the right and negative if it's moving to the left.  

\begin{Lemma} \label{L:fourBounces}
Suppose that $m$ and $n$ are odd and that $m<n$. Then for $0 < 2k < m$, the four bounces at $(0, 2k)$, $(m-2k,m)$, $(n, m-2k)$, and $(n-m+2k, 0)$ all have the same sign.
\end{Lemma}

\begin{proof}
The two bounces at $(0, 2k)$ and $(m-2k,m)$ have the same sign because they are directly connected to each other, see Figure~\ref{F:symmetry1}. 
\begin{figure}[ht]
\begin{center}
\begin{tikzpicture} [scale=0.5]

\draw [lightgray] (0,0) grid (11,7);
\draw (0,0) rectangle (11,7);

\draw [pathColor, thick] (1.5, 2.5) -- (0,4) -- (3,7) -- (4.5,5.5);
\draw[pathColor, thick] (9.5,4.5)--(11, 3)--(8,0)--(6.5,1.5);

\node[left] at (0,4) {$(0,2k)$};
\node[above] at (3,7) {$(m-2k,m)$};
\node[right] at (11,3) {$(n,m-2k)$};
\node[below] at (8,0) {$(n-m+2k,0)$};
 
\foreach \x/\y in {}
 \draw[pathColor, thick] (\x +0.7,\y+0.3) -- (\x +0.7,\y+0.7) -- (\x +0.3,\y+0.7);  
  
  \foreach \x/\y in {10/2, 8/0,0/4, 2/6}
  \draw[pathColor, thick] (\x +0.3,\y+0.7) -- (\x +0.3,\y+0.3) -- (\x +0.7,\y+0.3);
  
  \foreach \x/\y in {10/3, 0/3}
  \draw[pathColor, thick] (\x +0.3,\y+0.3) -- (\x +0.7,\y+0.3) -- (\x +0.7,\y+0.7);
  
  \foreach \x/\y in {7/0, 3/6}
  \draw[pathColor, thick] (\x +0.3,\y+0.3) -- (\x +0.3,\y+0.7) -- (\x +0.7,\y+0.7);

\end{tikzpicture}
\end{center}
\caption{Provided $m$ and $n$ are odd, the four bounces indicated will have the same sign.}
\label{F:symmetry1}
\end{figure}
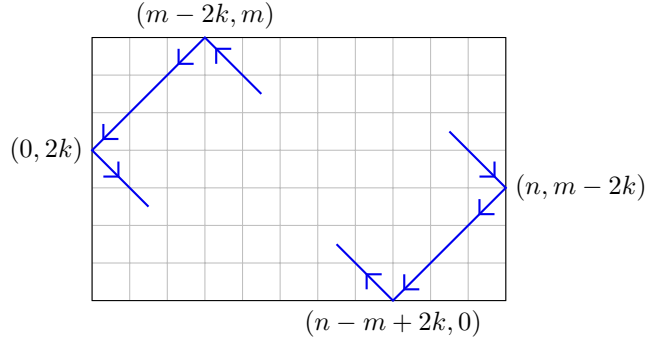
If the ball travels from $(0, 2k)$ to $(m-2k,m)$, both bounces are positive, while if it travels in the opposite direction as in the figure, both are negative. The conclusion holds also for ``zero-bounces'' when $m$ and $n$ have a common factor: If one of the bounces doesn't occur, neither does the other.
For the same reason the two bounces at $(n, m-2k)$ and $(n-m+2k, 0)$ must have the same sign.

But actually all four bounces must have the same sign because of a rotational symmetry: Since $m$ and $n$ are odd, the path will end at the upper right corner $(n,m)$. Therefore if we rotate the path $180$ degrees around the center of the rectangle and reverse its direction, we get back the same path. This symmetry shows that the bounces at $(0,2k)$ and at $(n,m-2k)$ must have the same sign, which means that all four signs are equal.
\end{proof}

\begin{Prop} \label{P:almostReciprocity}
If $m$ and $n$ are odd and $m<n$, then
\begin{equation} \label{almostReciprocity}
\left(\frac{m}{n}\right)\left(\frac{n}{m}\right) = \left(\frac{m}{n-m}\right)
\end{equation}
\end{Prop}

\begin{proof} 
We can assume that $m$ and $n$ are relatively prime, since otherwise both sides of the equation are zero.
The left hand side $(m|n)(n|m)$ is the product of the signs of the bounces along the $x$-axis and along the $y$-axis in $m$ by $n$ arithmetic billiards. By Lemma~\ref{L:fourBounces} we can cancel all the bounces along the $y$-axis against the bounces to the right of the point $(n-m,0)$ on the $x$-axis.
We can also cancel the bounce at the point $(n-m,0)$ since it is directly connected to the endpoint $(n, m)$ and is therefore positive.
The only bounces that remain are those along the $x$-axis to the left of the point $(n-m,0)$. 

The next observation is that the movement of the ball inside the leftmost ``corridor'' of width $n-m$ is the same as if the ball had been confined to this area by a vertical wall at $x=n-m$, see Figure~\ref{F:reciprocity2}. Whenever the ball reaches the line $x=n-m$ from the left, it will make a tour of the rightmost $m$ by $m$ square and return to the same point with a $90$ degree change of direction, continuing as if it had bounced off a wall at $x=n-m$. 

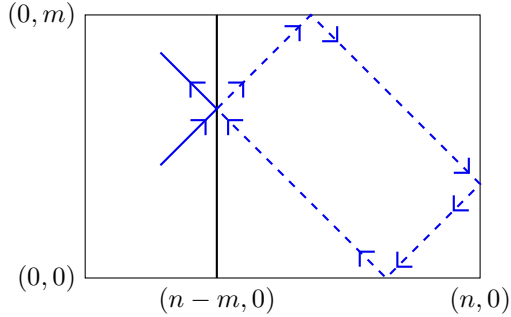
\begin{figure}[ht]
\begin{center}
\begin{tikzpicture} [scale=0.5]

\draw (0.5,0) rectangle (11,7);
\draw [thick] (4,0)--(4,7);

\draw [pathColor, thick] (2.5,3) -- (4,4.5) -- (2.5,6);
\draw [pathColor, thick, dashed] (4,4.5) -- (6.5,7)--(11,2.5)--(8.5,0)--cycle;
  
 \foreach \x/\y in {3/3.5, 4/4.5, 5.5/6}
  \draw[pathColor, thick] (\x +0.7,\y+0.3) -- (\x +0.7,\y+0.7) -- (\x +0.3,\y+0.7);  

  \foreach \x/\y in {10/1.5, 8.5/0}
  \draw[pathColor, thick] (\x +0.3,\y+0.7) -- (\x +0.3,\y+0.3) -- (\x +0.7,\y+0.3);

  \foreach \x/\y in {6.5/6, 10/2.5}
  \draw[pathColor, thick] (\x +0.3,\y+0.3) -- (\x +0.7,\y+0.3) -- (\x +0.7,\y+0.7);

  \foreach \x/\y in {7.5/0, 4/3.5, 3/4.5}
  \draw[pathColor, thick] (\x +0.3,\y+0.3) -- (\x +0.3,\y+0.7) -- (\x +0.7,\y+0.7);
  
\node[below] at (4,0) {$(n-m,0)$};
\node[left] at (0.5,0) {$(0,0)$};
\node[left] at (0.5,7) {$(0,m)$};
\node[below] at (11,0) {$(n,0)$};

\end{tikzpicture}
\end{center}
\caption{Each time the ball reaches the line $x=n-m$ from the left, it will make a tour through the rightmost $m$ by $m$ square and come back to the same point, continuing as if it had bounced against a wall at $x=n-m$.}
\label{F:reciprocity2}

\end{figure}
 
Therefore the bounces that haven't been canceled from the left hand side $(m|n)(n|m)$ of equation~\eqref{almostReciprocity} have the same signs as those that occur in an $m$ by $n-m$ rectangle, and their product is given by the right hand side of \eqref{almostReciprocity}.
\end{proof}

It remains to show that the right hand side of equation \eqref{almostReciprocity} 
depends only on the congruence classes modulo 4 of the numbers involved. In the following lemma we let $d$ stand for the difference $n-m$, which is even. 

\begin{Lemma} \label{L:mod4}
Suppose $m$ and $d$ are relatively prime with $m$ odd and $d$ even. Then 
\[ 
\left(\frac{m}{d}\right) = 
\begin{cases}
1, & \text{\rm if $d\equiv 2$ (mod $4$),}\\
(-1)^{(m-1)/2}, & \text {\rm if $d\equiv 0$ (mod $4$).} 
\end{cases}
\]
\end{Lemma}

\begin{proof}
When $d$ is even and $m$ is odd, the path of the ball in arithmetic billiards on an $m$ by $d$ rectangle will end at the lower right corner $(d,0)$. Therefore it has another symmetry: If we reflect the path in the vertical line $x=d/2$ and reverse its direction, we get back the same path. 
This means that a bounce at a point $(2k,0)$ corresponds to a bounce of the same sign at the point $(d-2k,0)$. 

If $d\equiv 2$ (mod 4) this is all there is to it: The bounces will cancel in pairs and $(m|d)=1$.

If on the other hand $d\equiv 0$ (mod 4), then there will be a bounce at the midpoint $(d/2,0)$ that doesn't cancel, and $(m|d)$ will be determined by the sign of this bounce. Again by the symmetry it follows that when the ball reaches the point $(d/2,0)$ it has travelled half-way to the endpoint $(d,0)$. This means that it has made exactly half of its left-right bounces. The total number of such bounces is $m-1$, and the direction of the ball at the moment it bounces at $(d/2,0)$ is therefore determined by the parity of $(m-1)/2$ in the way required.  
\end{proof} 

\begin{proof} [Proof of Quadratic Reciprocity]
To finish the proof of Theorem~\ref{T:qr}, suppose that $m$ and $n$ are odd and relatively prime. We can assume without loss of generality that $m<n$. Since $n-m$ is even, Proposition~\ref{P:almostReciprocity} and Lemma~\ref{L:mod4} show that
\[
\left(\frac{m}{n}\right)\left(\frac{n}{m}\right) = \left(\frac{m}{n-m}\right) 
= \begin{cases}
1, & \text{\rm if $n-m\equiv 2$ (mod $4$),}\\
(-1)^{(m-1)/2}, & \text {\rm if $n-m\equiv 0$ (mod $4$).} 
\end{cases}
\]
Here it's easy to check that the right hand side is equal to $(-1)^{(m-1)(n-1)/4}$ as required: The only time it equals $-1$ is when $m\equiv n \equiv 3$ (mod 4).
\end{proof}

\section{Parity checkers and reciprocity again} \label{S:checkers}
We return to parity checkers to see how it relates to arithmetic billiards, and to give a slightly different proof of quadratic reciprocity.

Let's say that a \emph{puzzle} in parity checkers consists of a set of pebbles on light squares of a checkerboard, and that a solution to such a puzzle is a set of checkers on dark squares such that the light squares with an odd number of neighbouring checkers are precisely those that have pebbles. Counting modulo~2, a puzzle is a system of linear equations where the constant terms are described by the configuration of pebbles.

A naive approach to a given puzzle is to go through the rows from top to bottom and for each light square, if the constraint on its number of neighbours isn't satisfied when we get to it, put a checker right under it. This way we can easily satisfy all constraints except those in the bottom row. This process is similar to so-called light chasing in the game \emph{Lights Out} \cite{lightsOut}. 

By linearity it follows that for a given shape of the checkerboard, if we can solve all puzzles that consist of just a single pebble in the bottom row, then we can solve all puzzles. 

\begin{Prop} \label{P:checkersInverse1}
For positive integers $m$ and $n$, if $\gcd(m,n)=1$ then every pebble puzzle on an $m-1$ by $n-1$ board has a unique solution. 
\end{Prop}

\begin{proof}
Suppose that on an $m-1$ by $n-1$ checkerboard we are faced with a puzzle with a single pebble located on square number $2k$ of the bottom row. 
Assuming that $\gcd(m,n)=1$, arithmetical billiards on an $m$ by $n$ rectangle will lead to a path that bounces at all points at even distance from the origin along the base, including the point $(2k,0)$. We color this path with two colors, one before the bounce at $(2k,0)$, and another after that bounce. Then we place checkers on the squares of the  checkerboard that correspond to points where the path crosses itself and the two parts that cross have distinct colors, as in Figure~\ref{F:twoColoring1}. 

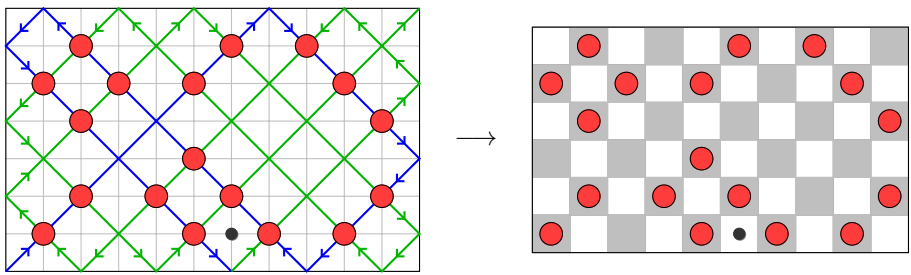
\begin{figure} [ht]
\begin{center}
\begin{tikzpicture} [scale=0.5]

\draw [lightgray] (0,0) grid (11,7);
\draw [black] (0,0) rectangle (11,7);

 \draw[firstColor, thick] (0,0)--(7,7)--(11,3)--(8,0)--(1,7)--(0,6)--(6,0);
 
 \draw[secondColor, thick] (6,0)--(11,5)--(9,7)--(2,0)--(0,2)--(5,7)--(11,1)--(10,0)
 --(3,7)--(0,4)--(4,0)--(11,7);    
 
 \foreach \x/\y in {0/0, 6/6}
  \draw[firstColor, thick] (\x +0.6,\y+0.4) -- (\x +0.6,\y+0.6) -- (\x +0.4,\y+0.6);  
  
  \foreach \x/\y in {10/2, 8/0, 0/6}
  \draw[firstColor, thick] (\x +0.4,\y+0.6) -- (\x +0.4,\y+0.4) -- (\x +0.6,\y+0.4);
  
 \foreach \x/\y in {7/6, 10/3, 0/5, 5/0}
 \draw[firstColor, thick] (\x +0.4,\y+0.4) -- (\x +0.6,\y+0.4) -- (\x +0.6,\y+0.6);
  
  \foreach \x/\y in {7/0, 1/6}
 \draw[firstColor, thick] (\x +0.4,\y+0.4) -- (\x +0.4,\y+0.6) -- (\x +0.6,\y+0.6); 

\foreach \x/\y in {6/0, 10/4, 0/2, 4/6, 4/0, 10/6}
  \draw[secondColor, thick] (\x +0.6,\y+0.4) -- (\x +0.6,\y+0.6) -- (\x +0.4,\y+0.6);  
 
  \foreach \x/\y in {8/6, 2/0, 10/0, 2/6, 0/4}
  \draw[secondColor, thick] (\x +0.4,\y+0.6) -- (\x +0.4,\y+0.4) -- (\x +0.6,\y+0.4);
  
 \foreach \x/\y in {5/6, 10/1, 0/3, 3/0}
 \draw[secondColor, thick] (\x +0.4,\y+0.4) -- (\x +0.6,\y+0.4) -- (\x +0.6,\y+0.6);
  
  \foreach \x/\y in {10/5, 9/6, 1/0,0/1, 9/0, 3/6}
 \draw[secondColor, thick] (\x +0.4,\y+0.4) -- (\x +0.4,\y+0.6) -- (\x +0.6,\y+0.6);
 
 \foreach \x/\y in {1/1, 1/5, 2/2, 2/4, 2/6, 3/5, 4/2, 5/1, 5/3, 5/5, 6/2, 6/6, 7/1, 8/6,
 9/1, 9/5, 10/2, 10/4}
 {
\filldraw[checkersColor] (\x, \y) circle (0.3cm);
\draw [black] (\x, \y) circle (0.3cm);
}

{
\filldraw [gray] (6.02, 0.98) circle (0.15cm);
\filldraw [black!80!white] (6, 1) circle (0.15cm);
}

\node at (12.5,3.5) {$\longrightarrow$}; 

\begin{scope} [xshift = 14cm, yshift=0.5cm] 

\draw [gray] (0,0) grid (10,6);

\foreach \x in {0, 2, 4, 6, 8}
  \foreach \y in {0, 2, 4}
  \filldraw[lightgray] (\x,\y) rectangle (\x+1,\y+1);
  
  \foreach \x in {1, 3, 5, 7, 9}
  \foreach \y in {1, 3, 5}
  \filldraw[lightgray] (\x,\y) rectangle (\x+1,\y+1);
 
 \draw (0,0) rectangle (10,6); 

\foreach \x/\y in {1/1, 1/5, 2/2, 2/4, 2/6, 3/5, 4/2, 5/1, 5/3, 5/5, 6/2, 6/6, 7/1, 8/6,
 9/1, 9/5, 10/2, 10/4}
 {
\filldraw[checkersColor] (\x-0.5, \y-0.5) circle (0.3cm);
\draw [black] (\x-0.5, \y-0.5) circle (0.3cm);
}
{
\filldraw [gray] (6.02-0.5, 0.98-0.5) circle (0.15cm);
\filldraw [black!80!white] (5.5, 0.5) circle (0.15cm);
}

\end{scope}
 
\end{tikzpicture}
\end{center}

\caption{To solve the 6 by 10 checkers puzzle with a single pebble at square 6 of the bottom row, we two-color the path in 7 by 11 arithmetic billiards, changing color at the bounce at $(6,0)$. Then we place checkers on those dark squares of the checkerboard that correspond to two-colored crossings.}
\label{F:twoColoring1}
\end{figure}

With this pattern of checkers, an ordinary light square on the checkerboard will have an even number of checkers in its immediate neighbourhood, since those checkers correspond to color-changes as we walk around one of the tilted squares formed by the path in arithmetic billiards. The one exception is the square with the pebble, since at the point $(2k,0)$ of the billiards path we have an artificial change of color that doesn't correspond to a checker.

This means that we have solved this single-pebble puzzle. By linearity, all puzzles with pebbles only in the bottom row can be solved, and by ``light chasing'', all puzzles on the $m-1$ by $n-1$ board can be solved.

Since we have assumed that $\gcd(m,n)=1$, at least one of $m-1$ and $n-1$ is even, which means that the board has equally many light and dark squares. Therefore it follows, by linearity or pigeon-hole reasoning, that each puzzle must have a unique solution. 
\end{proof}

For completeness we establish the following converse of Proposition~\ref{P:checkersInverse1}, even though it will not be needed in the following.

\begin{Prop} \label{P:checkersInverse2}
If $m$ and $n$ have a common factor, then there are nontrivial solutions to the empty pebble puzzle on an $m-1$ by $n-1$ checkerboard, and by linearity no pebble puzzle can have a unique solution.
\end{Prop}

\begin{proof}
When $m$ and $n$ have a common factor, the path in arithmetic billiards exits before time $mn$ and in particular visits the point $(1,1)$ only once. We now color the diagonals according to whether they are part of the path or not as in Figure~\ref{F:twoColoring2}. Putting checkers at the two-colored crossings gives a nontrivial solution to the empty pebble-puzzle.
\end{proof} 

\begin{figure} [ht]
\begin{center}
\begin{tikzpicture} [scale=0.55]

\draw [lightgray] (0,0) grid (9,6);
\draw [black] (0,0) rectangle (9,6);

 \draw[firstColor, thick] (0,0)--(6,6)--(9,3)--(6,0)--(0,6);
 \draw[secondColor, thick] (0,2)--(4,6)--(9,1)--(8,0)--(2,6)--(0,4)--(4,0)--(9,5)--(8,6)--(2,0)--cycle;    
 
 \foreach \x/\y in {0/0, 5/5}
  \draw[firstColor, thick] (\x +0.6,\y+0.4) -- (\x +0.6,\y+0.6) -- (\x +0.4,\y+0.6);  

  \foreach \x/\y in {8/2, 6/0}
  \draw[firstColor, thick] (\x +0.4,\y+0.6) -- (\x +0.4,\y+0.4) -- (\x +0.6,\y+0.4);
    
 \foreach \x/\y in {6/5, 8/3}
 \draw[firstColor, thick] (\x +0.4,\y+0.4) -- (\x +0.6,\y+0.4) -- (\x +0.6,\y+0.6);
  
  \foreach \x/\y in {5/0, 0/5}
 \draw[firstColor, thick] (\x +0.4,\y+0.4) -- (\x +0.4,\y+0.6) -- (\x +0.6,\y+0.6); 
 
 \foreach \x/\y in {1/1, 2/2, 4/4, 5/5, 7/5, 8/4, 8/2, 7/1, 5/1, 4/2, 2/4, 1/5}
 {
\filldraw[checkersColor] (\x, \y) circle (0.3cm);
\draw [black] (\x, \y) circle (0.3cm);
}

\node at (10.5,3) {$\longrightarrow$}; 

\begin{scope} [xshift = 12cm, yshift=0.5cm] 
\draw [gray] (0,0) grid (8,5);

\foreach \x/\y in {1/1, 1/3, 1/5, 2/2, 2/4, 3/1, 3/3, 3/5, 4/2, 4/4, 5/1, 5/3, 5/5, 6/2, 6/4, 7/1, 7/3, 7/5, 8/2, 8/4}
  \filldraw[lightgray] (\x-1,\y-1) rectangle (\x,\y);

\draw (0,0) rectangle (8,5);

 \foreach \x/\y in {1/1, 2/2, 4/4, 5/5, 7/5, 8/4, 8/2, 7/1, 5/1, 4/2, 2/4, 1/5}
 {
\filldraw [checkersColor] (\x-0.5, \y-0.5) circle (0.3cm);
\draw [black] (\x-0.5, \y-0.5) circle (0.3cm);
}

\end{scope}
 
\end{tikzpicture}
\end{center}
\caption{When $m$ and $n$ have a common factor, we can find a nontrivial solution to the empty pebble-puzzle on an $m-1$ by $n-1$ checkerboard.} 
\label{F:twoColoring2}
\end{figure}
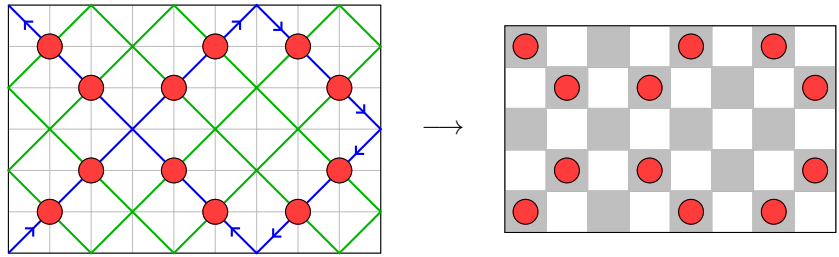

\begin{Remark*}
By famous work of Pieter Kasteleyn \cite{Kasteleyn}, the number of domino tilings of a rectangle can be computed as the determinant of a signed version of the adjacency matrix of the bipartite structure of light and dark squares. If we are only interested in the parity of the number of domino tilings, the signs are not necessary and the number will be odd or even according to whether the adjacency matrix is invertible or not modulo 2. But inverting the adjacency matrix means solving a general pebble puzzle on the given checkerboard. A corollary of Propositions~\ref{P:checkersInverse1} and \ref{P:checkersInverse2} is therefore that the number of domino tilings of an $m-1$ by $n-1$ rectangle is odd precisely when $\gcd(m,n)=1$. 

Remarkably, while I was working on this article, Yuhi Kamio, Junnosuke Koizumi and Toshihiko Nakazawa published the paper \cite{dominos} which gives a proof of quadratic reciprocity by relating the Jacobi symbol $(m|n)$ to a sum over domino tilings of an $m-1$ by $n-1$ rectangle! 
\end{Remark*}

Next we connect puzzles with pebbles in the bottom row to the Zolotarev symbol, which we still take to be defined by the signs of bounces in arithmetic billiards as in Section~\ref{S:reciprocity}.

\begin{Thm} \label{T:checkersLegendre}
Suppose that $m$ and $n$ are relatively prime, and consider the puzzle where we put pebbles on the light squares of the bottom row of an $m-1$ by $n-1$ checkerboard. Let $s$ be the number of checkers in the solution. Then 
\[ \left(\frac{m}{n}\right) = (-1)^s. \]
\end{Thm}

\begin{proof}
We show that the number of checkers in the solution to a single-pebble puzzle with a pebble on square $2k$ of the bottom row as in Figure~\ref{F:twoColoring1} is even if the corresponding bounce at the point $(2k,0)$ in arithmetic billiards is positive, and odd if that bounce is negative.

The path in $m$ by $n$ arithmetic billiards crosses itself $(m-1)(n-1)/2$ times, and these crossings correspond to the dark squares of an $m-1$ by $n-1$ checkerboard. 

In the solution to the single-pebble puzzle, the number of checkers has the same parity as the number of times the blue part of the path goes trough a crossing: We can imagine that we follow the blue part of the path and place a checker every time we reach a crossing, and finally remove the checkers from the crossings that have two.

The length of the blue path is a multiple of $2m$, and in particular even. The number of times it passes a crossing is that length minus the number of times it bounces, including the bounce at the endpoint $(2k,0)$. To figure out the parity of that number of bounces in turn, we notice that the number of bounces against the top of the rectangle is the same as the number of bounces against the bottom. Moreover, the number of left and right bounces are equal if the ball is moving from left to right as it bounces at $(2k,0)$, while if it's moving from right to left it has bounced one more time at the right than at the left side of the rectangle. 
\end{proof}

Using Theorem~\ref{T:checkersLegendre} we can give a slightly different proof of reciprocity based on parity checkers. Let $m$ and $n$ be odd and relatively prime. On an $m-1$ by $n-1$ checkerboard, let $s$ be the number of checkers in the solution to the bottom row puzzle, and let $t$ be the number of checkers in the solution to the puzzle given by instead putting pebbles on the light squares of the leftmost column. 
From Theorem~\ref{T:checkersLegendre} we know that $(m|n) = (-1)^s$ and $(n|m) = (-1)^t$.

Now let $u$ be the number of checkers in the solution to the puzzle with pebbles both along the bottom row and along the leftmost column. By linearity we can obtain the solution to that puzzle by superposing the solutions to the two former puzzles. Therefore $u\equiv s+t$ (mod 2), and consequently
\[ \left(\frac{m}{n}\right)\left(\frac{n}{m}\right) = (-1)^{s+t} = (-1)^u.\]
On the other hand a simple pattern emerges, showing that
\[ u = \frac{(m-1)(n-1)}4.\]

\begin{figure} [ht]
\begin{center}
\begin{tikzpicture} [scale=0.53]

\draw [gray] (0,0) grid (10,6);

\foreach \x in {0, 2, 4, 6, 8}
  \foreach \y in {0, 2, 4}
  \filldraw[lightgray] (\x,\y) rectangle (\x+1,\y+1);
  
  \foreach \x in {1, 3, 5, 7, 9}
  \foreach \y in {1, 3, 5}
  \filldraw[lightgray] (\x,\y) rectangle (\x+1,\y+1);
 
 \draw (0,0) rectangle (10,6); 
 
 \foreach \x/\y in {1/1, 1/3, 1/5, 2/6, 3/3, 4/2, 4/4, 4/6, 5/3, 6/6, 7/1, 7/3, 7/5, 9/3,
 9/5, 10/2, 10/6}
 {
\filldraw [checkersColor] (\x-0.5, \y-0.5) circle (0.3cm);
\draw [black] (\x-0.5, \y-0.5) circle (0.3cm);
}

\foreach \x in {1,3,5,7,9}
{
\filldraw [gray] (\x+0.52, 0.48) circle (0.15cm);
\filldraw [black!80!white] (\x+0.5, 0.5) circle (0.15cm);
}

\node at (11.25, 3) {$+$}; 

\begin{scope} [xshift = 12.5cm] 

\draw [gray] (0,0) grid (10,6);

\foreach \x in {0, 2, 4, 6, 8}
  \foreach \y in {0, 2, 4}
  \filldraw[lightgray] (\x,\y) rectangle (\x+1,\y+1);
  
  \foreach \x in {1, 3, 5, 7, 9}
  \foreach \y in {1, 3, 5}
  \filldraw[lightgray] (\x,\y) rectangle (\x+1,\y+1);
 
 \draw (0,0) rectangle (10,6); 

\foreach \y in {1,3,5}
{
\filldraw [gray] (0.52, \y+0.48) circle (0.15cm);
\filldraw [black!80!white] (0.5, \y+0.5) circle (0.15cm);
}

 \foreach \x/\y in {1/1, 1/3, 1/5, 2/2, 2/4, 3/3, 5/3, 6/2, 6/4, 7/1, 7/3, 7/5, 8/2,
 8/4, 8/6, 9/3, 9/5, 10/4}
 {
\filldraw [checkersColor] (\x-0.5, \y-0.5) circle (0.3cm);
\draw [black] (\x-0.5, \y-0.5) circle (0.3cm);
}

\end{scope}

\begin{scope} [xshift = 6.25cm, yshift=-8cm] 

\node at (-1.5, 3) {$=$};

\draw [gray] (0,0) grid (10,6);

\foreach \x in {0, 2, 4, 6, 8}
  \foreach \y in {0, 2, 4}
  \filldraw[lightgray] (\x,\y) rectangle (\x+1,\y+1);
  
  \foreach \x in {1, 3, 5, 7, 9}
  \foreach \y in {1, 3, 5}
  \filldraw[lightgray] (\x,\y) rectangle (\x+1,\y+1);
 
 \draw (0,0) rectangle (10,6); 

\foreach \y in {1,3,5}
{
\filldraw [gray] (0.52, \y+0.48) circle (0.15cm);
\filldraw [black!80!white] (0.5, \y+0.5) circle (0.15cm);
}

\foreach \x in {1,3,5,7,9}
{
\filldraw [gray] (\x+0.52, 0.48) circle (0.15cm);
\filldraw [black!80!white] (\x+0.5, 0.5) circle (0.15cm);
}

 \foreach \x/\y in {2/2, 2/4, 2/6, 4/2, 4/4, 4/6, 6/2, 6/4, 6/6, 8/2, 8/4, 8/6,
 10/2, 10/4, 10/6}
 {
\filldraw [checkersColor] (\x-0.5, \y-0.5) circle (0.3cm);
\draw [black] (\x-0.5, \y-0.5) circle (0.3cm);
}

\end{scope}

\end{tikzpicture}

\label{F:finalFigure}
\end{center}

\end{figure}

\end{document}